\documentclass[twoside,a4paper,12pt,centertags]{amsart}
\usepackage{amsmath,amssymb,verbatim,vmargin}
\usepackage[bookmarks=true]{hyperref}   
\usepackage{color}

\theoremstyle{plain}
\newtheorem{thm}{Theorem}[section]
\newtheorem{lem}[thm]{Lemma}

\newtheorem{prop}[thm]{Proposition}

\theoremstyle{definition}
\newtheorem{defn}[thm]{Definition}

\mathchardef\semic="303B
\newcommand{\R}{{\mathbf R}}
\newcommand{\C}{{\mathbf C}}
\newcommand{\Z}{{\mathbf Z}}
\newcommand{\mH}{{\mathcal H}}

\newcommand{\mL}{{\mathcal L}}

\DeclareMathOperator{\re}{Re}

\newcommand{\sett}[2]{ \{ #1 \, \semic \, #2 \} }

\newcommand{\nul}{\textsf{N}}
\newcommand{\ran}{\textsf{R}}

\newcommand{\clos}[1]{\overline{#1}}

\newcommand{\sgn}{\text{{\rm sgn}}}
\newcommand{\barint}{\mbox{$ave \int$}}
\newcommand{\divv}{{\text{{\rm div}}}}
\newcommand{\curl}{{\text{{\rm curl}}}}
\newcommand{\esssup}{\text{{\rm ess sup}}}

\newcommand{\ta}{{\scriptscriptstyle \parallel}}
\newcommand{\no}{{\scriptscriptstyle\perp}}
\newcommand{\pd}{\partial}

\newcommand{\loc}{\text{{\rm loc}}}
\newcommand{\tN}{\widetilde N_*}

\newcommand{\tE}{\widetilde E}

\newcommand{\bx}{{\bf x}}
\newcommand{\by}{{\bf y}}
\newcommand{\bz}{{\bf z}}
\newcommand{\tB}{\widetilde B}

\newcommand{\dlp}{{\mathcal D}}
\newcommand{\tdlp}{\widetilde{\mathcal D}}
\newcommand{\slp}{{\mathcal S}}
\newcommand{\tslp}{\widetilde{\mathcal S}}

\def\barint_#1{\mathchoice
            {\mathop{\vrule width 6pt
height 3 pt depth -2.5pt
                    \kern -8.8pt
\intop}\nolimits_{#1}}%
            {\mathop{\vrule width 5pt height
3 pt depth -2.6pt
                    \kern -6.5pt
\intop}\nolimits_{#1}}%
            {\mathop{\vrule width 5pt height
3 pt depth -2.6pt
                    \kern -6pt
\intop}\nolimits_{#1}}%
            {\mathop{\vrule width 5pt height
3 pt depth -2.6pt
          \kern -6pt \intop}\nolimits_{#1}}}

\usepackage{color}

\definecolor{gr}{rgb}   {0.,   0.8,   0. } 
\definecolor{bl}{rgb}   {0.,   0.5,   1. } 
\definecolor{mg}{rgb}   {0.7,  0.,    0.7} 
          
\makeatletter
\@namedef{subjclassname@2010}{%
  \textup{2010} Mathematics Subject Classification}
\makeatother

\begin{document}

\title[Layer potentials beyond SIOs]
{Layer potentials beyond singular integral operators}
\author[Andreas Ros\'en]{Andreas Ros\'en$\,^1$}
\thanks{$^1\,$Formerly Andreas Axelsson}
\address{Andreas Ros\'en, Matematiska institutionen, Link\"opings universitet, 581 83 Lin\-k\"oping, Sweden}
\email{andreas.rosen@liu.se}

\begin{abstract}
We prove that the double layer potential operator and the gradient of the single layer potential operator
are $L_2$ bounded for general second order divergence form systems.
As compared to earlier results, our proof shows that the bounds for the layer potentials are independent of well posedness for the Dirichlet problem and of De Giorgi--Nash local estimates.
The layer potential operators are shown to depend holomorphically on the coefficient matrix $A\in L_\infty$, showing uniqueness of the extension of the operators beyond singular integrals.
More precisely, we use functional calculus of differential operators with non-smooth coefficients
to represent the layer potential operators as bounded Hilbert space operators.
In the presence of Moser local bounds, in particular for real scalar equations and systems that are small perturbations of real scalar equations, these operators are shown to  be the usual singular integrals. 
Our proof gives a new construction of fundamental solutions to divergence form systems, valid also in dimension $2$.
\end{abstract}



\thanks{Supported by Grant 621-2011-3744 from the Swedish research council, VR}
\maketitle

\section{Introduction}

This paper concerns the classical boundary value problems for divergence form second order elliptic systems
$$
   \sum_{j=1}^m\divv A^{ij} \nabla u^j=0,\qquad i=1,\ldots, m,
$$
for a vector valued function $u=(u^j)_{j=1}^m$
on the upper half space 
$\R^{1+n}_+:= \sett{(t,x)\in \R\times \R^n}{t>0}$, $n,m\ge 1$,
with boundary data in $L_2(\R^n)$.
In general, we only assume that the coefficients $A= (A^{ij})_{i,j=1}^m$ are uniformly bounded and accretive.
(Accretivity, or more precisely strict accretivity, is defined in \eqref{eq:accre} below.)
Unless otherwise stated, we assume that $A(t,x)=A(x)$ is independent of the transversal direction $t$.
However, we do not assume that $A$ is real or symmetric. 

By scalar coefficients, or equation, we mean that $A^{ij}=0$ for $i\ne j$.
For technical reasons we consider systems where the functions $u^j$ are complex-valued, and thus $A^{ij}(t,x)\in \mL(\C^{1+n})$.
However, working at the level of systems of equations of arbitrary size, complex coefficients are no more general than real coefficients. Indeed, using the relation $\C=\R^2$ we see that any system of equations with complex coefficients of size $m$ can be viewed as a system of equations with real coefficients of size $2m$. 

For an $1+n$-dimensional vector $f$,
we let $f_\no$ denote the normal/vertical part (identified with the corresponding scalar coordinate),
and write $f_\ta$ for the tangential/horizontal part.
Similarly, we write $\nabla_\ta$, $\divv_\ta$ and $\curl_\ta$ for the differential operators acting only in the tangential/horizontal variable $x$.
To ease notation, we use the Einstein summation convention throughout this paper.
Sometimes we shall even suppress indices $i,j$.

A classical method for solving the Dirichlet problem is to solve the associated double layer potential equation at the boundary $\R^n$.
In our framework, the method is the following.
Let $\Gamma_{(t,x)}= (\Gamma^{ij}_{(t,x)})_{i,j=1}^m$ be the fundamental solution for $\divv A^* \nabla$ in $\R^{1+n}$ with pole at $(t,x)$, that is $\divv (A^{ji})^* \nabla \Gamma^{jk}_{(t,x)}= \begin{cases}\delta_{(t,x)}, & i=k, \\ 0, & i\ne k\end{cases}$,
and let
$\pd_{\nu_{A^*}}\Gamma_{(t,x)}^{ij}:= ((A^{ki})^* \nabla \Gamma_{(t,x)}^{kj})_\no$ denote its (inward) conormal derivative.

Given a function $h:\R^n\to\C^m$ on the boundary, define the function
$$
  \dlp_t h^i(x):= \int_{\R^n} \big(-\pd_{\nu_{A^*}}\Gamma^{ji}_{(t,x)}(0,y) , h^j(y)\big)\, dy, \qquad (t,x)\in\R^{1+n}_+,
$$
where $-\pd_{\nu_{A^*}}$ is the outward conormal derivative.
The function $u(t,x):= \dlp_t h(x)$ then solves the equation $\divv A \nabla u=0$ in $\R^{1+n}_+$, and has boundary trace
$$
  \dlp h^i(x):= \lim_{t\to 0^+}  \int_{\R^n} \big(-\pd_{\nu_{A^*}}\Gamma^{ji}_{(t,x)}(0,y) , h^j(y)\big)\,dy.
$$
Finding the solution $u$ with Dirichlet data $\varphi:\R^n\to\C^m$ on the boundary, then amounts to solving the double layer equation 
$$
   \dlp h= \varphi
$$
for $h$, which then gives the solution $u(t,x)= \dlp_t h(x)$.
In the case of smooth coefficients $A$, it is well known that the operator $\dlp$ is well defined and is $\tfrac 12 I$ plus an integral operator. 
For general systems with non-smooth coefficients, as considered in this paper, the double layer potential 
operator  $\dlp$ is beyond the scope of singular integral theory.

Similarly, the single layer potential is used to solve the Neumann problem. See Section~\ref{sec:singlelayer}. In this introduction, we focus on the double layer potential and the Dirichlet problem.

During the last years, new results on boundary value problems for more general non-smooth divergence form systems have been proved.
In particular, there have been two seemingly different developments, one based on singular integrals (S) and one based on functional calculus (F).
The purpose of this paper is to demonstrate that the singular integral operators used in (S) actually are special cases of the abstract operators used in (F).

\begin{itemize}
\item[(\rm S)] 
In the paper \cite{AAAHK}  by Alfonseca, Auscher, Axelsson, Hofmann and Kim, it was proved in particular
that boundedness and invertibility of the layer potential operators for coefficients $A_0$ implies  boundedness and invertibility of the layer potential operators for coefficients $A$ whenever $\|A-A_0\|_\infty$ is small, depending on $A_0$. Here $A_0$ and $A$ are assumed to be scalar and complex, and such that De Giorgi--Nash local H\"older estimates hold for solutions to these equations.
Boundedness here includes square function estimates.
This boundedness and invertibility result was shown to hold for real symmetric coefficients, and the result was also known for coefficients of block form and for constant coefficients.

During the writing of this paper, Hofmann, Kenig, Mayboroda and Pipher~\cite{HKMP} have proved $L_p$ well posedness, for some $p<\infty$ depending on $A$, of the Dirichlet problem for general scalar equations with real and $t$-independent coefficients. 
From this they deduce, in \cite[Cor. 1.25]{HKMP}, boundedness in $L_2$ (but not invertibility) of the layer potentials for general real scalar equations and small complex perturbations of such, 
by inspection of the proofs in \cite{AAAHK}.

After submission of this paper, Grau de la Herr\'an and Hofmann~\cite{GdlHH} proved $L_2$ estimates for layer potentials with complex coefficients, assuming De Giorgi--Nash local estimates.
\item[{\rm (F)}]
Auscher, Axelsson and McIntosh~\cite{AAM} proved that 
the $L_2$ Dirichlet (and Neumann) problem is well posed for systems with coefficients $A$ which are small $L_\infty$ perturbations of Hermitian, constant or block form coefficients.
Instead of the double layer potential operator $\dlp$ above, this used an operator $\tdlp$ on $L_2(\R^n)$ defined by  functional calculus from an underlying differential operator on $\R^n$.
More precisely, this used a self-adjoint first order differential operator $D$ and a transformed multiplication operator $B$ formed point wise from the coefficients $A$,
to construct a solution
$$
  u^i(t,x)=\tdlp_t h^i(x) :=( b_t(BD) h)^i_\perp (x),\qquad (t,x)\in \R^{1+n}_+,
$$
where the function $b_t(z):= \begin{cases} e^{-tz},& \re z>0, \\ 0, & \re z<0, \end{cases}$ is applied to the operator $BD$
by functional calculus. Here $BD$, and therefore $b_t(BD)$, acts on $\C^{m(1+n)}$-valued functions $h$ on $\R^n$.
\end{itemize}

Both works \cite{AAAHK, AAM}  build on harmonic analysis developed for the solution of the Kato square root problem by
Auscher, Hofmann, Lacey, McIntosh and Tchamitchian~\cite{AHLMcT}.
However, the approach (F) is more general. 
On the one hand (S) uses De Giorgi--Nash local H\"older estimates, which holds for real scalar equations, and small $L_\infty$ perturbations of such, but not for general $A$. 
On the other hand, (F) proves that $\tdlp_t$ in fact is $L_2$-bounded for any $t$-independent and uniformly bounded and accretive coefficients $A$;
it is only invertibility of $\tdlp:= \lim_{t\to 0^+}\tdlp_t$ which may fail.
Note that (F) does not use De Giorgi--Nash local bounds at all.

Unlike $\tdlp$, the definition of the double layer potential operator $\dlp$ require the existence of a fundamental solution to $\divv A^*\nabla$. For divergence form systems, such fundamental solutions were constructed by Hofmann and Kim~\cite{HK} under the hypothesis that solutions to $\divv A\nabla u=0$ and $\divv A^*\nabla u=0$ satisfy 
{\em De Giorgi--Nash local H\"older estimates}.
That solutions to $\divv A\nabla u=0$ satisfy such estimates means that
\begin{equation}  \label{eq:dGN}
   \esssup_{\by, \bz\in B(\bx;R), \by\ne \bz} \frac {|u(\by)-u(\bz)|}{|\by-\bz|^\alpha} \lesssim
    R^{-\alpha-(1+n)/2}\left( \int_{B(\bx;2R)} |u|^2 \right)^{1/2} 
\end{equation}
holds whenever $u$ is a weak solution to $\divv A\nabla u=0$ in $B(\bx; 2R)\subset\R^{1+n}$, for some $\alpha>0$.
It is known that \eqref{eq:dGN} is equivalent to the gradient estimate
\begin{equation}  \label{eq:Morrest}
  \int_{B(\bx, r)} |\nabla u |^2 \lesssim (r/R)^{n-1+2\mu}\int_{B(\bx, R)} |\nabla u |^2, \qquad 0<r<R,
\end{equation}
for all weak solutions $u$ to $\divv A\nabla u=0$ in $B(\bx; R)\subset\R^{1+n}$, for some $\mu>0$.

It is known that \eqref{eq:dGN}, or equivalently \eqref{eq:Morrest}, holds for all divergence form systems $\divv A \nabla u=0$ where $A$ is real and scalar, and small $L_\infty$-perturbations of such ($t$-independence of $A$ is not needed here).
Estimates \eqref{eq:dGN} and \eqref{eq:Morrest} also imply the {\em Moser local boundedness estimate}
\begin{equation}  \label{eq:Mos}
   \esssup_{\by\in B(\bx;R)} |u(\by)| \lesssim
    R^{-(1+n)/2}\left( \int_{B(\bx;2R)} |u|^2 \right)^{1/2} 
\end{equation}
whenever $\divv A \nabla u=0$ in $B(\bx; 2R)\subset\R^{1+n}$.
We refer to \cite[Sec. 2]{HK} for further explanation of these results.

At the 8th International Conference on Harmonic Analysis and Partial Differential Equations at El Escorial 2008, S. Hofmann formulated as an open problem whether (F) as a special case implies the result (S).
Our main result in this paper is that this is indeed the case, as $\dlp=\tdlp$ whenever $\dlp$ is defined. More precisely, we prove the following.

\begin{thm}  \label{thm:1}
  Let $n,m\ge 1$, and let $A=A(x)\in L_\infty(\R^n;\mL(\C^{m(1+n)}))$ be $t$-independent
  and accretive  in the sense that there exists $\kappa>0$ such that
\begin{equation}   \label{eq:accre}
  \re  \int_{\R^n} ( A^{ij}(x) f^j(x), f^i(x) ) dx \ge \kappa \int_{\R^n} |f(x)|^2 dx,
\end{equation}
for all $f\in L_2(\R^n;\C^{m(1+n)})$ such that $\curl_\ta f_\ta=0$.

Assume that whenever $u$ is a weak solution to $\divv A \nabla u=0$ in a ball $B(\bx; 2R)$, 
$u$ is almost everywhere equal to a continuous function and the 
Moser local boundedness estimate \eqref{eq:Mos} holds.
Then there exists a fundamental solution $\Gamma_{(t,x)}\in W^1_{1,\loc}(\R^{1+n};\C^{m^2})$ to $\divv A^*\nabla$ with estimates
$$
  \int_{|y-x|>R}  |\nabla \Gamma_{(t, x)}(s,y)|^2 dy\lesssim (R+|s-t|)^{-n},
$$
for all $R>0$, $t,s\in\R$ and $x\in\R^n$.
Moreover
\begin{equation}   \label{eq:reprofdlp}
 \int_{\R^n}  \big(-\pd_{\nu_{A^*}}\Gamma^{ji}_{(t,x)}(0,y) , h^j(y)\big)\,dy=
( b_t(BD) h)^i_\perp (x)
\end{equation}
holds for almost all $(t,x)\in \R^{1+n}_+$ and all scalar functions $h\in L_2(\R^n;\C^m)$.
The right hand side is defined in Section~\ref{sec:funccalc}. In particular, we here identify $h$ with a normal vector field $h\in L_2(\R^n;\C^{m(1+n)})$.
\end{thm}

This theorem allows us to transfer known results for the double layer potential operator 
$$
  \tdlp_t^A h^i= \tdlp_t h^i :=( b_t(BD) h)^i_\perp,\qquad t>0,
$$
defined through functional calculus, to the double layer potential operator
$$
 \dlp_t^A h^i= \dlp_t h^i:= \int_{\R^n} \big(-\pd_{\nu_{A^*}}\Gamma^{ji}_{(t,\cdot)}(0,y) , h^j(y)\big)\, dy, \qquad t>0,
$$
defined classically as an integral operator.
The following is a list of such known results for $\tdlp_t^A$, which therefore also hold for 
$\dlp_t^A=  \tdlp_t^A$ under the hypothesis of Theorem~\ref{thm:1}.
These results for $\tdlp^A_t$ follow by inspection of the proof of \cite[Thm. 2.3 and 2.2]{AAM}, and extends the results for $\dlp^A_t$ from \cite{AAAHK, HKMP}.

\begin{itemize}
\item We have estimates
$$
\sup_{t>0} \|\tdlp_t h\|_2^2+  \int_0^\infty \|\pd_t \tdlp_t h \|_2^2\, tdt + \| \tN (\tdlp_t h)\|_2^2\lesssim \|h\|^2,
$$
for any system with bounded and accretive coefficients $A$, where the
modified non-tangential maximal function $\tN$ is defined in Section~\ref{sec:funccalc}.
In particular, the implicit constant in this estimate depends only on $\|A\|_\infty$, $\kappa_A$, $n, m$, but not on the De Giorgi--Nash--Moser constants.
In presence of Moser local boundedness estimates of solutions, $\tN$ can be replaced by the 
usual point wise non-tangential maximal function.

\item
For any system with bounded and accretive coefficients $A$, the operators $\tdlp_t$ converge strongly in $L_2$ and there exists an $L_2(\R^n;\C^m)$ bounded operator $\tdlp$ such that
$$
  \lim_{t\to 0^+}\|\tdlp_t h- \tdlp h\|_2=0,\qquad\text{for all } h\in L_2(\R^n;\C^m).
$$

\item
The map
$$
  \{ \text{accretive } A\in L_\infty(\R^n;\mL(\C^{m(1+n)}))\} \ni A\mapsto \tdlp^A\in \mL(L_2(\R^n;\C^{m}))
$$
is a holomorphic map between Banach spaces. 
In particular, $\tdlp^A\in \mL(L_2(\R^n;\C^{m}))$ depends locally Lipschitz continuously on $A\in L_\infty(\R^n;\mL(\C^{m(1+n)}))$, and therefore invertibility of $\tdlp^A$ is stable under small $L_\infty$ perturbations of $A$.

\item
The operator $\tdlp^A\in \mL(L_2(\R^n;\C^{m}))$ is invertible when $A$ is Hermitian, $(A^{ij})^*= A^{ji}$, when 
$A$ is constant, $A(x)= A$, and when $A$ is of block form, $A^{ij}_{\no\ta}=0= A^{ij}_{\ta\no}$.

\item
It is also known that $\tdlp^A$ is not invertible for many $A$, even for real and scalar (but non-symmetric) coefficients $A$ in the plane, $n=1$.
A counter example was found in \cite[Thm 3.2.1]{KKPT} among the coefficients
$$
  A(x)= \begin{bmatrix} 1 & k\sgn(x) \\ -k\sgn(x) & 1 \end{bmatrix}.
$$
Note that $\dlp^A= \tdlp^A$ for all these coefficients by Theorem~\ref{thm:1}.
It was shown in \cite{AxNon} that $\tdlp^A$ is not invertible for these coefficients when $k=1$.
Moreover, from \cite{AxNon} and \cite[Rem. 5.4]{AAM} it follows that $\tdlp^A$ is invertible for these coefficients when $k\ne 1$, but that the coefficients with $k>1$ are disconnected from the identity $A=I$
by the set of coefficients for which $\tdlp^A$ is not invertible.
\end{itemize}

In the process of proving Theorem~\ref{thm:1}, we also give a new construction of fundamental solutions to divergence form systems. As compared to \cite{HK}, this works also in dimension $2$, and constructs the gradient fundamental solution directly using functional calculus, taking \eqref{eq:reprofdlp} as a definition of the fundamental solution.
Extending this construction to $t$-dependent coefficients, we prove the following result.
Note that we formulate this result in dimension $n$, not $1+n$.

\begin{thm}   \label{thm:2}
  Let $n\ge 2$, $m\ge 1$. Assume that
$A_0\in L_\infty(\R^{n};\mL(\R^{n}))$ are real and scalar coefficients, identified with a matrix acting component-wise 
on $f\in \C^{mn}$, which are accretive in the sense that there exists $\kappa>0$ such that
$$
  \re(A_0(x)f,f)\ge \kappa |f|^2,\qquad\text{for all } f\in \C^{mn}, x\in\R^n.
$$

Then there exists $\epsilon>0$ such that whenever $A\in L_\infty(\R^{n};\mL(\C^{mn}))$ is such that
$\esssup_{x\in\R^n}|A(x)-A_0(x)|<\epsilon$, then there exists a fundamental solution $\Gamma_{x}$ to $\divv A \nabla$,
i.e. a function $\Gamma_{x}\in W_{1,\loc}^1(\R^{n};\C^{m^2})$ such that $\divv A^{ij}\nabla\Gamma^{jk}_{x}=\begin{cases} \delta_{x}, & i= k, \\ 0, & i\ne k,\end{cases}$ in distributional sense, with estimates
\begin{equation}   \label{eq:keyestimateofgradfund}
  \int_{R<|y-x|<2R}  |\nabla \Gamma_{ x}(y)|^2\, dy\lesssim R^{2-n},
\end{equation}
for all $R>0$ and $x\in\R^n$.
\end{thm}

From the gradient estimate \eqref{eq:keyestimateofgradfund}, we deduce point wise estimates of 
$\Gamma_x$ in Section~\ref{sec:tdep}. This section also contains the proof of Theorem~\ref{thm:2}, 
which builds on the proof of Theorem~\ref{thm:1}, which is in Section~\ref{sec:fundindep}.
Sections~\ref{sec:funccalc}, \ref{sec:GreenHalf} and \ref{sec:GreenLip} contains the details of the construction of the fundamental solution for $t$-independent coefficients, which uses the Green's formula from Definition~\ref{def:fundsolpm}. Half of this identity yields the representation formula \eqref{eq:reprofdlp} for the double layer potential operator. By a duality argument we also derive corresponding results for the gradient of the single layer potential operator in Section~\ref{sec:singlelayer}.

\section{Functional calculus for divergence form equations}   \label{sec:funccalc}

In this section, we explain the method of functional calculus (F) for the $L_2$ Dirichlet problem
for the equation $\divv A \nabla u=0$ in $\R^{1+n}_+$.
We assume in this section that the coefficients $A\in L_\infty(\R^{1+n};\mL(\C^{m(1+n)}))$
are $t$-independent and accretive in the sense of \eqref{eq:accre}.
  
  Recall from complex analysis the following two relations between harmonic functions and analytic functions in $\C=\R^2$: (a) $u$ is harmonic if and only if $f=\nabla u$ is anti-analytic, that is divergence- and curl-free, and (b) 
$u$ is harmonic if and only if there exists an analytic function $v$ with $\re v=u$.
In this section, we generalize this result to solutions to $\divv A \nabla u=0$ in $\R^{1+n}$, following \cite{AAM, AA1}.
Following the notation from these papers, we shall suppress indices $i,j=1,\ldots, m$ in this section.

(a)
If $\divv A\nabla u=0$, write $f= [f_\no, f_\ta]^t := [\pd_{\nu_A} u, \nabla_\ta u]^t$, where 
 $[a, b]^t:= \begin{bmatrix}a \\ b \end{bmatrix}$.
Decomposing the matrix $A$ as
$$
  A(x) = \begin{bmatrix} A_{\no\no}(x) & A_{\no\ta}(x) \\ A_{\ta\no}(x) & A_{\ta\ta}(x)  \end{bmatrix},
$$
we have the conormal derivative $\pd_{\nu_A} u := A_{\no\no} \pd_t u+ A_{\no\ta}\nabla_\ta u$,
or inversely $\pd_t u= A_{\no\no}^{-1}(f_\no-A_{\no\ta} f_\ta)$.
In terms of $f$, the equation for $u$ becomes
$$
  \pd_t f_\no + \divv_\ta \Big( A_{\ta\no}A_{\no\no}^{-1}(f_\no-A_{\no\ta} f_\ta) + A_{\ta\ta}f_\ta  )  \Big)=0.
$$
The condition that $f$ is the conormal gradient of a function $u$, determined up to constants,
can be expressed as the curl-free condition
$$
  \begin{cases}
    \pd_t f_\ta = \nabla_\ta \Big( A_{\no\no}^{-1}(f_\no-A_{\no\ta} f_\ta) \Big), \\
    \curl_\ta f_\ta=0.
  \end{cases}
$$ 
In vector notation, we equivalently have 
$$
  \pd_t  \begin{bmatrix}f_\no \\ f_\ta \end{bmatrix}+
  \begin{bmatrix} 0 & \divv_\ta \\ -\nabla_\ta & 0  \end{bmatrix}
   \begin{bmatrix} A_{\no\no}^{-1} & -A_{\no\no}^{-1}A_{\no\ta} \\ A_{\ta\no}A_{\no\no}^{-1} & A_{\ta\ta}- A_{\ta\no}A_{\no\no}^{-1}A_{\no\ta}  \end{bmatrix} \begin{bmatrix}f_\no \\ f_\ta \end{bmatrix}=0,
$$
together with the constraint $\curl_\ta f_\ta=0$.
Define
$$
  D:=  \begin{bmatrix} 0 & \divv_\ta \\ -\nabla_\ta & 0  \end{bmatrix}
  \qquad \text{and}\qquad 
  B:= \begin{bmatrix} A_{\no\no}^{-1} & -A_{\no\no}^{-1}A_{\no\ta} \\ A_{\ta\no}A_{\no\no}^{-1} & A_{\ta\ta}- A_{\ta\no}A_{\no\no}^{-1}A_{\no\ta}  \end{bmatrix},
$$
so that the equation becomes 
\begin{equation}   \label{eq:DB}
   \pd_t f+ DB f=0
\end{equation} 
together with the constraint $f\in \mH:= \clos{\ran (D)}= \sett{f\in L_2}{\curl_\ta f_\ta=0}$ for each fixed $t>0$.
(Here and below, $\ran(\cdot)$ and $\nul(\cdot)$ denote range and null space of an operator.)
This equation for $f$, which is an $L_2(\R^n;\C^{m(1+n)})$ vector-valued ODE in $t$, can be viewed as
a generalized Cauchy--Riemann system.

\begin{defn}
The {\em conormal gradient} of $u$ is the vector field 
$$
   \nabla_A u:= \begin{bmatrix} \pd_{\nu_A} u \\ \nabla_\ta u  \end{bmatrix},
$$ 
where $\pd_{\nu_A} u= (A\nabla u)_\no$ is the (inward relative $\R^{1+n}_+$) conormal derivative.
\end{defn}

(b)
Another Cauchy--Riemann type system related to $\divv A\nabla u=0$ is
$$
  \pd_t v+ BD v=0,
$$
where $D$ and $B$ have been swapped.
Applying $D$ to this equation yields $(\pd_t+ DB)(Dv)=0$,
so 
$$
  f:= Dv = [\divv_\ta v_\ta, -\nabla_\ta v_\no]
$$
is the conormal gradient of a solution $u$ to $\divv A\nabla u=0$.
Looking at $f_\ta$, we see that we should set
$$
  u:= -v_\no.
$$
Then $\nabla_\ta u= f_\ta$. Moreover
$$
  \pd_t u= -\pd_t v_\no= (BDv)_\no= (Bf)_\no= A_{\no\no}^{-1}(f_\no-A_{\no\ta}\nabla_\ta u),
$$
so that $\pd_{\nu_A}u= f_\no$.
Thus the equation
\begin{equation}   \label{eq:BD}
   \pd_t v+ BD v=0
\end{equation} 
for $v= [-u, v_\ta]^t$ implies that $u$ solves 
$\divv A\nabla u=0$. The vector-valued function $v_\ta$ can be viewed as as a set of generalized conjugate functions to $u$.

\begin{defn}
  A {\em conjugate system} for $u$ is a vector field $v$ solving $\pd_t v+BDv=0$ such that
$$
   u= -v_\no.
$$
\end{defn}

We now consider the closed and unbounded operators $DB$ and $BD$ in the Hilbert space 
$L_2= L_2(\R^n;\C^{m(1+n)})$. 
Here $D$ is a non-injective (if $n\ge 2$) self-adjoint operator with
$$
   \clos{\ran(D)}=\mH\qquad\text{and}\qquad \nul(D)= \mH^\perp,
$$
whereas $B$ is a bounded and accretive multiplication operator just like $A$.
Indeed, in \cite{AAM} it was noted that the transform
  $$
   A=\begin{bmatrix} A_{\no\no} & A_{\no\ta} \\ A_{\ta\no} & A_{\ta\ta}  \end{bmatrix}\mapsto
   \hat A:= \begin{bmatrix} A_{\no\no}^{-1} & -A_{\no\no}^{-1}A_{\no\ta} \\ A_{\ta\no}A_{\no\no}^{-1} & A_{\ta\ta}- A_{\ta\no}A_{\no\no}^{-1}A_{\no\ta}  \end{bmatrix}
  $$
  has the following properties.
\begin{itemize}
\item[{\rm (i)}] If $A$ is accretive, then so is $\hat A$.
\item[{\rm (ii)}]  If $\hat A=B$, then $\widehat B= A$.
\item[{\rm (iii)}] If $\hat A=B$, then $\widehat{A^*}= NB^* N$, where
$N:= \begin{bmatrix} -I & 0 \\ 0 & I\end{bmatrix}$ is the reflection operator for vectors across $\R^n$.
\end{itemize}

As $B$ is bounded and accretive, we have
$$
   \omega:=\sup_{f\in \mH\setminus\{0\}} |\arg(Bf, f)]<\pi/2.
$$
The operators $DB$ and $BD$ both have spectrum contained in the double sector
$$
  S_{\omega-}\cup \{0\}\cup S_{\omega+},
$$
where $S_{\omega+}= \sett{\lambda\in\C\setminus\{0\}}{|\arg \lambda|\le \omega}$ and $S_{\omega-}:= -S_{\omega+}$.
There are decompositions of $L_2$ into closed complementary (but in general non-orthogonal) spectral subspaces associated with these three parts of the spectrum. For $DB$ we have
$$
  L_2= E^-_A L_2\oplus B^{-1}\mH^\perp \oplus  E_A^+ L_2
$$
and for $BD$ we have
$$
L_2 = \tE^-_A L_2\oplus \mH^\perp\oplus  \tE_A^+ L_2.
$$
Note that for $DB$ we have $\clos{\ran(DB)}= \mH= E^-_A L_2 \oplus  E_A^+ L_2$ and $\nul(DB)= B^{-1}\mH^\perp$, whereas for $BD$ we have
$\clos{\ran(BD)}=B\mH=  \tE^-_A L_2\oplus  \tE_A^+ L_2$ and $\nul(BD)= \mH^\perp$.
The proof of the fact that the the projections $E^\pm_A$ and $\tE_A^\pm$ associated with these splittings are bounded uses harmonic analysis from the solution of the Kato square root problem.

Important in this paper are the following intertwining and duality relations.

\begin{prop}    \label{prop:dualityandintertw}
We have well-defined isomorphisms 
$$
  B: E^\pm_A L_2\to \tE^\pm_A L_2,
$$
and closed and injective maps with dense domain and range
$$
  D: \tE^\pm_A L_2\to E^\pm_A L_2.
$$
We also have a duality
$$
   ( E_{A^*}^\mp, N \tE_A^\pm ),
$$
that is the map $E_{A^*}^\mp L_2\to (\tE_A^\pm L_2)^*$, mapping $g\in E_{A^*}^\mp L_2$
to the functional $\tE_A^\pm L_2\ni f\mapsto (g, Nf)\in \C$, is an isomorphism.
\end{prop}

\begin{proof}
  The intertwining by $B$ is a consequence of associativity $B(DB)= (BD)B$, 
  the intertwining by $D$ is a consequence of associativity $D(BD)= (DB)D$,
  and the duality is a consequence of the duality
$$
  (g, N(BD)f)= (g, -(NBN)DNf)= ((-D\widehat {A^*})g, Nf).
$$
\end{proof}

To solve Equation~\eqref{eq:DB} for $f\in \mH$, we note that $DB$ restricts to an operator 
in $E_A^\pm L_2$ with spectrum
$$
  \sigma(DB|_{E_A^\pm L_2})\subset S_{\omega\pm}.
$$
Thus $e^{-tDB}f$ is well defined for $f\in E^+_AL_2$ if $t\ge 0$ and for $f\in E^-_AL_2$ if $t\le 0$.

The following result was proved in \cite{AA1}.
Here the modified non-tangential maximal function of a function $f$ in $\R^{1+n}_\pm$ is the function
$\tN f$ on $\R^n$ defined by
$$
  \tN f(x):= \sup_{\pm t>0} |t|^{-(1+n)/2}\|f\|_{L_2(W(t,x))}, \qquad x\in \R^n,
$$
where the Whitney regions are $W(t,x):= \sett{(s,y)}{c_0^{-1}<s/t< c_0, |y-x|<c_1 |t|}$, for some fixed constants $c_0>1, c_1>0$.
Also, here and below, we write $f_t(x):= f(t,x)$.

\begin{prop}    \label{prop:Xsol}
  Let $f_0\in E_A^\pm L_2$ and define
$$
  f(t,x):= (e^{-tDB}f_0)(x), \qquad \pm t>0, x\in \R^n.
$$
Then 
\begin{itemize}
\item[{\rm (i)}]  $f= [\pd_{\nu_A}u, \nabla_x u]^t$ for a weak solution $u$ to $\divv A \nabla u=0$ in $\R^{1+n}_\pm$, unique up to constants,
\item[{\rm (ii)}] $\pm(0,\infty)\ni t \mapsto f_t\in L_2$ is continuous, with $\lim_{t\to 0^\pm} f_t= f_0$ and $\lim_{t\to \pm\infty} f_t=0$ in $L_2$ sense, and
\item[{\rm (iii)}] we have estimates
$$
  \|f_0\|^2_2\approx \sup_{\pm t>0} \|f_t\|^2_2 \approx \iint_{\R^{1+n}_\pm}|\pd_t f(t,x)|^2 t dtdx\approx
  \int_{\R^n} |\tN(f)|^2 dx.
$$
\end{itemize}
Conversely, if $u$ is any weak solution to $\divv A \nabla u=0$ in $\R^{1+n}_\pm$, with estimate
$\|\tN(f)\|_2<\infty$, or $\sup_{\pm t>0} \|f_t\|_2<\infty$, of the conormal gradient $f=[\pd_{\nu_A}u, \nabla_x u]$, then there exists $f_0\in E_A^\pm L_2$ such that $f(t,x)= (e^{-tDB}f_0)(x)$ almost everywhere in $\R^{1+n}_\pm$. 
\end{prop}

Similar results apply to Equation~\eqref{eq:BD}.
The following result was proved in \cite{AA1}.

\begin{prop}    \label{prop:Ysol}
  Let $v_0\in \tE_A^\pm L_2$ and define
$$
  v(t,x):= (e^{-tBD}v_0)(x), \qquad \pm t>0, x\in \R^n.
$$
Then 
\begin{itemize}
\item[{\rm (i)}]  $u:= -v_\no$ is a weak solution $u$ to $\divv A \nabla u=0$ in $\R^{1+n}_\pm$,
\item[{\rm (ii)}] $\pm(0,\infty)\ni t\mapsto v_t\in L_2$ is continuous, with $\lim_{t\to 0^\pm} v_t= v_0$ and $\lim_{t\to \pm\infty} v_t=0$ in $L_2$ sense, and
\item[{\rm (iii)}] we have estimates
$$
  \|v_0\|^2_2\approx \sup_{\pm t>0} \|v_t\|^2_2 \approx \iint_{\R^{1+n}_\pm}|\nabla u(t,x)|^2 t dtdx\approx
  \int_{\R^n} |\tN(v)|^2 dx.
$$
\end{itemize}
Conversely, if $u$ is any weak solution to $\divv A\nabla u=0$ in $\R^{1+n}_\pm$, with estimate
$\iint_{\R^{1+n}_\pm}|\nabla u(t,x)|^2 t dtdx <\infty$, then there exists $v_0\in \tE_A^\pm L_2$ and a  constant $c\in \C^m$ such that $u(t,x)= -(e^{-tBD}v_0)_\perp(x)+c$ almost everywhere in $\R^{1+n}_\pm$. 
\end{prop}

\section{Green's formula on the half space}    \label{sec:GreenHalf}

Recall that for the Laplace operator, that is the special case $A=I$ and $m=1$, we have the fundamental solution
$$
  \Phi(t,x)= 
  \begin{cases}
     \frac {-1}{(n-1)\sigma_n} \big(t^2+|x|^2\big)^{-(n-1)/2},  & n\ge 2, \\
     \frac 1{2\pi}\ln\sqrt{t^2+x^2}, & n=1,
  \end{cases}
$$
with pole $(0,0)$, where $\sigma_n$ denotes the area of the unit sphere in $\R^{1+n}$.
We note that 
$$
 \nabla\Phi(t,x)=  \frac 1{\sigma_n}
  \frac{(t,x)}{(t^2+|x|^2)^{(n+1)/2}},
$$
for $n\ge 1$.

In this section, we construct a fundamental solution to more general divergence form operators
$\divv A \nabla$ using functional calculus.
We assume in Sections~\ref{sec:GreenHalf}, \ref{sec:GreenLip} and \ref{sec:fundindep} that the coefficients $A\in L_\infty(\R^{1+n};\mL(\C^{m(1+n)}))$
are $t$-independent, accretive in the sense of \eqref{eq:accre} and that solutions to 
$\divv A\nabla u=0$ satisfy the Moser local boundedness estimate \eqref{eq:Mos}.

To explain the definition, we start with the following formal calculation.
Assume that $\Gamma= (\Gamma^{ij}_{(t_0,x_0)}(t,x))_{i,j=1}^m$ is a fundamental solution to $\divv A^* \nabla$, that is 
$$
   \divv (A^{ki})^* \nabla \Gamma^{kj}_{(t_0,x_0)}=\begin{cases} \delta_{(t_0, x_0)},& i=j, \\ 0, & i\ne j.\end{cases}
$$
Assume that $t_0>0$ and that $u$ solves $\divv A \nabla u=0$ in $\R^{1+n}_+$.
With appropriate estimates of $u$ and $\Gamma^*$, Green's formula shows that
$$
   u^i(t_0,x_0)= \int_{\R^n} \Big( \big( \Gamma^{ji}_{(t_0,x_0)}(0,x), \pd_{\nu_A} u^j(0,x) \big)- \big( \pd_{\nu_{A^*}}\Gamma^{ji}_{(t_0,x_0)}(0,x), u^j(0,x)\big)  \Big) dx,
$$
where the conormal derivative is $\pd_{\nu_A}u^j=  (A^{jk}\nabla u^k)_\no$.
Now let $v$ be a conjugate system for $u$ so that $u^j=-v^j_\no$ and $\pd_{\nu_A}u^j=\divv_\ta v^j_\ta$.
Then by integration by parts, we obtain
$$
u^i(t_0,x_0)= -\int_{\R^n} \big(\nabla_{A^*}\Gamma^{ji}_{(t_0,x_0)}(0,x), N v^j(0,x)\big) dx,
$$
where the conormal gradient is $\nabla_{A^*}\Gamma^{ji}= [\pd_{\nu_{A^*}}\Gamma^{ji}, \nabla_\ta\Gamma^{ji} ]^t$.
More generally, it follows in this way that if $v_0\in \tE_A^+ L_2$, then
$$
  \int_{\R^n} \big(\nabla_{A^*}\Gamma^{ji}_{(t_0,x_0)}(0,x), N v^j_0(x) \big) dx
  =
  \begin{cases}
     (e^{-t_0 BD}v_0)^i_\no(x_0), & t_0>0, \\
     0, & t_0<0,
  \end{cases}
$$
and if $v_0\in \tE_A^- L_2$, then
$$
  \int_{\R^n} \big(\nabla_{A^*}\Gamma^{ji}_{(t_0,x_0)}(0,x), N v^j_0(x) \big) dx
  =
  \begin{cases}
     0, & t_0>0, \\
     -(e^{-t_0 BD}v_0)^i_\no(x_0), & t_0<0.
  \end{cases}
$$
We now reverse this argument, taking  these four formulae as definition. 
From the Moser local boundedness estimate \eqref{eq:Mos}, it follows that 
$$
  |u(t_0, x_0)|\lesssim \tN u(x) \qquad \text{for } |x-x_0|<c_1|t_0|/2.
$$
Thus
$$
  |(e^{-t_0 BD}v_0)_\no(x_0)|\lesssim |t_0|^{-n/2}\|v_0\|_2,
$$
uniformly for $v_0\in \tE_A^\pm L_2$ and $\pm t_0>0$.
Proposition~\ref{prop:Xsol} and the duality from Proposition~\ref{prop:dualityandintertw} 
enable us to make the following construction.

\begin{defn}   \label{def:fundsolpm}  
  For $(t_0, x_0)\in \R_+^{1+n}$ and $i=1,\ldots, m$, let $\Gamma^i_{(t_0,x_0)}=(\Gamma^{ji}_{(t_0,x_0)})_j$ be the, unique up to constants, weak solution to $\divv A^* \nabla \Gamma^i_{(t_0,x_0)}=0$ in $\R^{1+n}_-$ such that
$$
  \int_{\R^n} \big(\nabla_{A^*} \Gamma^{ji}_{(t_0,x_0)}(0,x), N v^j_0(x) \big) dx= (e^{-t_0 BD}v_0)^i_\no(x_0), 
$$
for all $v_0\in\tE^+_A L_2$.

  For $(t_0, x_0)\in \R_-^{1+n}$ and $i=1,\ldots, m$, let $\Gamma^i_{(t_0,x_0)}=(\Gamma^{ji}_{(t_0,x_0)})_j$ be the, unique up to constants, weak solution to $\divv A^* \nabla \Gamma^i_{(t_0,x_0)}=0$ in $\R^{1+n}_+$ such that
$$
  \int_{\R^n} \big(\nabla_{A^*} \Gamma^{ji}_{(t_0,x_0)}(0,x), N v^j_0(x) \big) dx= -(e^{-t_0 BD}v_0)^i_\no(x_0), 
$$
for all $v_0\in\tE^-_A L_2$.
\end{defn}

Some straightforward observations are the following.

\begin{lem}  \label{lem:translinv}
  For $t_0,a>0$, there is a constant $c\in \C^{m^2}$ such that
$$
  \Gamma_{(t_0,x_0)}(t-a,x)= \Gamma_{(t_0+a,x_0)}(t, x)+c 
$$
for almost all $(t,x)\in \R^{1+n}_-$.
  Similarly, for fixed $t_0<0$, $a>0$, there is a constant $c\in \C^{m^
  2}$ such that
$$
  \Gamma_{(t_0,x_0)}(t+a,x)= \Gamma_{(t_0-a,x_0)}(t, x)+c 
$$
for almost all $(t,x)\in \R^{1+n}_+$.

Furthermore, there are estimates
$\|\nabla \Gamma_{(t_0,x_0)}(t,\cdot)\|_2\lesssim |t-t_0|^{-n/2}$.
\end{lem}

\begin{proof}
Fix $t_0,a>0$ and consider the functions $\nabla_{A^*}\Gamma^i_{(t_0,x_0)}(t-a,\cdot)= e^{-(t-a) D\tB}\nabla_{A^*}\Gamma^i_{(t_0,x_0)}(0,\cdot)$ and 
$\nabla_{A^*}\Gamma^i_{(t_0+a,x_0)}(t,\cdot)= e^{-t D\tB}\nabla_{A^*}\Gamma^i_{(t_0+a,x_0)}(0,\cdot)$ in $E^-_{A^*}L_2$, where $\tB:= N B^* N$.
We have
\begin{multline*}
 \int_{\R^n} \big(\nabla_{A^*}\Gamma^{ji}_{(t_0,x_0)}(t-a,x), N v^j_0(x) \big) dx=
  \int_{\R^n} \big(\nabla_{A^*}\Gamma^{ji}_{(t_0,x_0)}(0,\cdot), N e^{(t-a) BD} v^j_0(x) \big) dx\\
 =(e^{-t_0 BD} e^{(t-a)BD} v_0)^i_\no(x_0) 
 = (e^{-(t_0+a) BD}e^{tBD} v_0)^i_\no(x_0) \\ 
 =  \int_{\R^n} \big(\nabla_{A^*}\Gamma^{ji}_{(t_0+a,x_0)}(0,\cdot), N e^{t BD} v^j_0(x) \big) dx=
  \int_{\R^n} \big(\nabla_{A^*}\Gamma^{ji}_{(t_0+a,x_0)}(t,x), N v^j_0(x) \big) dx
\end{multline*}
for all $v_0\in\tE^+_A L_2$, and therefore $\nabla_{A^*}\Gamma^{i}_{(t_0,x_0)}(t-a,x)= \nabla_{A^*}\Gamma^{i}_{(t_0+a,x_0)}(t,x)$.

The proof for $t_0<0$ is similar.
The estimate of $\|\nabla \Gamma_{(t_0,x_0)}(t,\cdot)\|_2$ follows from Proposition~\ref{prop:dualityandintertw} and the bound $|t_0|^{-n/2}$ of the functional $v_0\mapsto (e^{-t_0 BD} v_0)^i_\no(x_0)$.
\end{proof}

Note that the translation invariance from Lemma~\ref{lem:translinv} enables us to define, for
any $(t_0,x_0)\in \R^{1+n}$, a weak solution $\Gamma_{(t_0,x_0)}(t,x)$ to $\divv A^* \nabla\Gamma_{(t_0,x_0)}=0$ in $\sett{(t,x)}{t\ne t_0}$, so that
$$
  \Gamma_{(t_0,x_0)}(t,x)= \Gamma_{(t_0+a,x_0)}(t+a, x).
$$
We shall prove in the following sections that for appropriate choices of constants, this defines a fundamental solution $\Gamma_{(t_0,x_0)}(t,x)$
to $\divv A^* \nabla$ on $\R^{1+n}$, that is that the traces at $t=t_0$ coincide except for a Dirac delta 
distribution at $(t,x)= (t_0,x_0)$.
Note that in this paper, except in Section~\ref{sec:tdep}, we only define the fundamental solution on $\R^{1+n}_+$ modulo constants.

\section{Green's formula on Lipschitz graph domains}  \label{sec:GreenLip}

In this section, we improve the estimate $\|\nabla_{A^*} \Gamma_{(t_0,x_0)}(t,\cdot)\|_2\lesssim |t-t_0|^{-n/2}$ away from $x_0$, and prove the following.

\begin{prop}  \label{prop:siobounds}
  We have for $R\ge 0$ and $t\ne t_0$ the estimate
$$
  \int_{|x-x_0|>R}  |\nabla_{A^*} \Gamma_{(t_0, x_0)}(t,x)|^2 dx\lesssim (R+|t-t_0|)^{-n}.
$$
\end{prop}

To prove this, we consider the graph 
$$
  \Sigma= \sett{(\gamma(x), x)}{x\in \R^n}
$$
of a Lipschitz function $\gamma: \R^n\to \R$. We assume $\gamma(x_0)=0$ and
write 
$$
  \sigma:= \nabla_\ta \gamma\in L_\infty(\R^n;\R^n).
$$
Recall the following consequence of the chain rule.

\begin{prop}  
  Let $\Omega\subset\R^{1+n}$ be an open set.
  Then $u$ is a weak solution to $\divv A \nabla u=0$ in 
  $\sett{(t+\gamma(x),x)}{(t,x)\in\Omega}$ if and only if 
  $$
    u_\sigma(t,x):= u(t+\gamma(x), x)
  $$
  is a weak solution to $\divv A_\sigma \nabla u_\sigma=0$ in $\Omega$.
  Here
  $$
    A_\sigma^{ij}:= \begin{bmatrix} 1 & -\sigma^t \\ 0 & I \end{bmatrix} A^{ij}
    \begin{bmatrix} 1 & 0 \\ -\sigma & I \end{bmatrix}
  $$
  has estimates $\|A_\sigma\|_\infty\lesssim (1+\|\sigma\|_\infty^2) \|A\|_\infty$
  and $\kappa_{A_\sigma}\gtrsim \kappa_A/(1+\|\sigma\|_\infty^2)$,
  where $\sigma^t$ denotes the transpose of the column vector $\sigma$.
\end{prop}

\begin{prop}   \label{prop:stokes}
  Fix $(t_0,x_0)\in \R^{1+n}_+$ and consider a Lipschitz graph $\Sigma$ as above such that
  $\gamma(x)\ge 0$ and $\gamma(x_0)=0$.
  Define $\Gamma= \Gamma_{(t_0,x_0)}$ for coefficients $A$, and define 
  $\Gamma_\sigma= \Gamma_{(t_0,x_0)}$ for coefficients $A_\sigma$, as in Definition~\ref{def:fundsolpm}.
  Then there is a constant $c\in\C^{m^2}$ such that
$$
  \Gamma(t, x)= \Gamma_\sigma(t-\gamma(x),x)+c
$$
for all $t<t_0$, $x\in\R^n$.
\end{prop}

\begin{proof}
(i)
 The function $\Gamma(t,x)$ is uniquely, up to constants, determined by the property that
 \begin{equation}    \label{eq:reprod1}
  \int_{\R^n} \big(\nabla_{A^*} \Gamma^{ji}(0,x), N v^j_0(x) \big) dx= (e^{-t_0 BD}v_0)^i_\no(x_0), 
\end{equation}
for all $v_0\in\tE^+_A L_2$.
By the intertwining $B: E^+_AL_2\to \tE^+_AL_2$ from Proposition~\ref{prop:dualityandintertw},
we can write 
$$
  v^j= (B\nabla_A u)^j= [ \pd_t u^j, (A^{jk}\nabla u^k)_\ta ]^t
$$
for a weak solution $u$ to $\divv A\nabla u=0$ in $\R^{1+n}_+$.
Then \eqref{eq:reprod1} reads
 \begin{equation}    \label{eq:reprod2}
  \int_{\R^n} \left(\nabla \Gamma^{ji}(0,x), 
  \begin{bmatrix} -A^{jk}_{\no\no} & 0 \\ 0 & A^{jk}_{\ta\ta}  \end{bmatrix}
   \nabla u^k(0,x) \right) dx= \pd_t u^i(t_0,x_0).
\end{equation}

(ii)
Now replace $A, \Gamma$ and $u$ by $A_\sigma$, $\Gamma_\sigma$ and 
$$
  u_\sigma(t,x):= u(t+\gamma(x),x).
$$
Then $u_\sigma$ is a weak solution to $\divv A^{ij}_\sigma \nabla u^j_\sigma=0$ for $t>-\gamma(x)$,
and in particular $\nabla_{A_\sigma} u_\sigma(0,\cdot)\in E^+_{A_\sigma}L_2$.
As in \eqref{eq:reprod2}, we have
 \begin{equation}    \label{eq:reprod3}
  \int_{\R^n} \left(\nabla \Gamma^{ji}_\sigma(0,x), 
  \begin{bmatrix} -(A^{jk}_\sigma)_{\no\no} & 0 \\ 0 & (A^{jk}_\sigma)_{\ta\ta}  \end{bmatrix}
   \nabla u^k_\sigma(0,x) \right) dx= \pd_t u^i_\sigma(t_0,x_0).
\end{equation}
Here $\Gamma_\sigma$ is a weak solution to $\divv (A^{ji}_\sigma)^* \nabla \Gamma^{jk}_\sigma=0$ for $t<t_0$.
Since $(A_\sigma)^*= (A^*)_\sigma$, the function
$$
  \widetilde \Gamma(t,x):= \Gamma_\sigma(t-\gamma(x),x)
$$
is a weak solution to $\divv (A^{ji})^* \nabla \widetilde\Gamma^{jk}=0$ for $t<t_0+\gamma(x)$.
Changing variables in \eqref{eq:reprod3}, we get
 \begin{equation}    \label{eq:reprod4}
  \int_{\R^n} \Big(\nabla \widetilde\Gamma^{ji}(\gamma(x),x), 
  \Lambda^{jk}(-1,\sigma(x))
   \nabla u^k(\gamma(x),x) \Big) dx= \pd_t u^i(t_0,x_0),
\end{equation}
where 
\begin{equation}   \label{eq:oneform}
  \Lambda^{jk}(\nu_0, \nu):= A^{jk}
  \begin{bmatrix} \nu_0 & 0 \\ \nu & 0  \end{bmatrix}+ 
    \begin{bmatrix} 0 & \nu^t  \\  0 & -\nu_0 \end{bmatrix} A^{jk}.
\end{equation}
This uses the chain rule
$$
  \nabla u^k_\sigma(t,x)=  \begin{bmatrix} 1 & 0 \\ \sigma(x) & I  \end{bmatrix}\nabla u^k(t+\gamma(x),x),
$$
and the calculation
$$
   \begin{bmatrix} 1 & 0 \\ \sigma & I  \end{bmatrix}^t
   \begin{bmatrix} -(A^{jk}_\sigma)_{\no\no} & 0 \\ 0 & (A^{jk}_\sigma)_{\ta\ta}  \end{bmatrix}
    \begin{bmatrix} 1 & 0 \\ \sigma & I  \end{bmatrix}
    =  \Lambda^{jk}(-1,\sigma(x)).
$$

(iii)
We now apply Stokes' theorem to the $1$-form $\approx$ $n$-form
$$
  (\nu_0,\nu)\mapsto \big(\nabla \widetilde\Gamma^{ji}(t,x), 
  \Lambda^{jk}(\nu_0,\nu)
   \nabla u^k(t,x) \big)
$$
on $\sett{(t,x)}{0<t<\gamma(x)}$.
Using \eqref{eq:oneform} and the product rule shows that its exterior derivative is
\begin{multline*}
   (\nabla \widetilde\Gamma^{ji}(t,x), 
  \Lambda^{jk}(\pd_t,\nabla_\ta)
   \nabla u^k(t,x) )
= (\divv (A^{jk})^*\nabla\widetilde\Gamma^{ji}, \pd_t u^k )+ ((A^{jk})^* \nabla\widetilde\Gamma^{ji}, \pd_t \nabla u^k) \\
+ ( \nabla_\ta\pd_t\widetilde\Gamma^{ji}-\pd_t\nabla_\ta\widetilde\Gamma^{ji}, (A^{jk}\nabla u^k)_\ta )
+ (\nabla\widetilde\Gamma^{ji}, [ \divv_\ta(A^{jk}\nabla u^k)_\ta, -\pd_t(A^{jk}\nabla u^k)_\ta ]^t)\\
=0+((A^{jk})^* \nabla\widetilde\Gamma^{ji}, \pd_t \nabla u^k) +0+ (\nabla\widetilde\Gamma^{ji}, -\pd_t(A^{jk}\nabla u^k))=0.
\end{multline*}
Thus, applying Stokes' theorem to \eqref{eq:reprod4} gives
\begin{multline*}
\pd_t u^i(t_0,x_0)= \int_{\R^n} \Big(\nabla \widetilde\Gamma^{ji}(0,x), 
  \Lambda^{jk}(-1,0)
   \nabla u^k(0,x) \Big) dx \\
   =  \int_{\R^n} \left(\nabla \widetilde\Gamma^{ji}(0,x), 
  \begin{bmatrix} -A^{jk}_{\no\no} & 0 \\ 0 & A^{jk}_{\ta\ta}  \end{bmatrix}
   \nabla u^k(0,x) \right) dx.
\end{multline*}
Comparing with \eqref{eq:reprod2} proves the proposition, by uniqueness of $\Gamma$.
\end{proof}

\begin{proof} [Proof of Proposition~\ref{prop:siobounds}]
Consider $\Gamma_{(t_0,x_0)}$.
By Lemma~\ref{lem:translinv}, we may assume that $t=0$.
Fix $R>t_0>0$ and apply Proposition~\ref{prop:stokes} with
$$
  \gamma(x):= 
  \begin{cases} |x|, & |x|<R, \\ R, & |x|>R. \end{cases}
$$
Then $\|A_\sigma\|_\infty\approx \|A\|_\infty$ and $\kappa_{A_\sigma}\approx \kappa_A$
and $\Gamma(0,x)= \Gamma_\sigma(-R,x)+c$ for $|x|>R$.
Thus the estimate from Lemma~\ref{lem:translinv} gives
\begin{multline*}
\int_{|x-x_0|>R}  |\nabla_{A^*} \Gamma(0,x)|^2 dx
= \int_{|x-x_0|>R}  |\nabla_{A^*} \Gamma_\sigma(-R,x)|^2 dx \\
\lesssim
 \|\nabla_{A_\sigma^*} \Gamma_\sigma(-R,\cdot)\|^2
\lesssim |R+t_0|^{-n}.
\end{multline*}
This proves the estimate for $t_0>t$.
The proof for $t_0<t$ is similar, using the analogue of 
Proposition~\ref{prop:stokes} for $t_0<t$.
\end{proof}

\section{Fundamental solution for $t$-independent coefficients}   \label{sec:fundindep}

In this section, we complete the proof of Theorem~\ref{thm:1}.
Fix $x_0\in \R^n$ and $i=1,\ldots, m$, and
define the vector field
$$
f^j(t,x) :=\nabla_{A^*}\Gamma^{ji}_{(0,x_0)}(t,x), \qquad t\ne 0.
$$
As in Section~\ref{sec:funccalc}, we suppress the index $j$.

\begin{prop}   \label{prop:solidestimates}
  For $R>0$, we have the estimate
$$
  \iint_{R<|(t,x)-(0,x_0)|<2R} |f(t,x)|^2 dtdx\lesssim R^{1-n}.
$$
In particular, for $1\le p<(n+1)/n$, we have $f\in L_p^\loc(\R^{1+n})$ and
$$
  \iint_{|(t,x)-(0,x_0)|<R} |f(t,x)|^p dtdx\lesssim R^{1-n(p-1)}.
$$
\end{prop}

\begin{proof}
  From Proposition~\ref{prop:siobounds}, we obtain the estimate
$$
  \iint_{R<|(t,x)-(0,x_0)|<2R} |f(t,x)|^2 dtdx\lesssim \int_0^{2R} \frac{dt}{(\max(R-t,0)+t)^n}\lesssim R^{1-n}.
$$
H\"older's inequality then gives the $L_p$-estimate after summing a geometric series.
\end{proof}

\begin{prop}    \label{prop:eqaccross0}
  We have that $f\in L_1^\loc(\R^{1+n})$ and, in $\R^{1+n}_+$ distributional sense,
$$
  ((\pd_t+ D\tB) f)^j= \begin{cases} (\delta_{(0,x_0)}, 0), & j=i, \\ 0, & j\ne i, \end{cases} 
$$
and $\curl_\ta f_\ta=0$.
\end{prop}

\begin{proof}
  That $\curl_\ta f^j_\ta=0$ is clear from the construction of $f$.
 To compute $(\pd_t+ D\tB) f$, we fix a test function $\phi\in C_0^\infty(\R^{1+n}; \C^{m(1+n)})$ and define
 $$
   I(x_0):= \iint_{\R^{1+n}} (f(t,x), -N(\pd_t + B(x)D)\phi(t,x)) dtdx.
 $$
 For $\epsilon>0$, let
 $$
   I_\epsilon(x_0):= \iint_{|t|>\epsilon} (f(t,x), -N(\pd_t + B(x)D)\phi(t,x)) dtdx.
 $$
 Since 
 $$
   |I(x_0)-I_\epsilon(x_0)|\lesssim \iint_{|t|<\epsilon, |x|<C} |f(t,x)| dtdx
 $$
 for some $C<\infty$ depending on $\phi$, we have that $I_\epsilon\to I$ uniformly in $x_0$ as $\epsilon\to 0$.
 By Definition~\ref{def:fundsolpm}, we have 
\begin{multline*}
   I_\epsilon(x_0)=- \int_\epsilon^\infty (e^{tBD} \tE^-_A (BD+\pd_t)\phi_t)^i_\no(x_0) dt +
  \int_{-\infty}^{-\epsilon} (e^{tBD} \tE^+_A (BD+\pd_t)\phi_t)^i_\no(x_0) dt \\
  = - \int_\epsilon^\infty \pd_t(e^{tBD} \tE^-_A \phi_t)^i_\no(x_0) dt +
  \int_{-\infty}^{-\epsilon} \pd_t(e^{tBD} \tE^+_A \phi_t)^i_\no(x_0) dt \\
  (e^{\epsilon BD} \tE^-_A \phi_\epsilon)^i_\no(x_0) +(e^{-\epsilon BD} \tE^+_A \phi_{-\epsilon})^i_\no(x_0)
\end{multline*}
 Therefore $I_\epsilon\in L_2$ and
 $$
   I_\epsilon\to (\tE^+_A \phi_0+ \tE^-_A\phi_0)^i_\no= (\phi_0)^i_\no
 $$
 in $L_2$ as $\epsilon\to 0$.
We have here used that $v_\no=0$ for $v\in \nul(BD)$. 
 
 We note that $I_\epsilon$ are continuous functions, since $A$ satisfies property (M), and converge uniformly to $I$. Thus $I$ is continuous, and it suffices to prove
 $$
   \int_K |I(x_0)-\phi(0,x_0)^i_\no|^2 dx=0
 $$
 for an arbitrary compact set $K$. 
 But this is clear since $I_\epsilon\to (\phi_0)^i_\no$ and $I_\epsilon\to I$ in $L_2(K)$.
This proves the proposition.
\end{proof}

\begin{proof}[Proof of Theorem~\ref{thm:1}]

Fix $(t_0,x_0)\in \R^{1+n}$, define as in Section~\ref{sec:GreenHalf} the function $\Gamma_{(t_0,x_0)}(t,x)$ for $t\ne t_0$. Without loss of generality, we can assume that $t_0=0$.
It follows from Propositions~\ref{prop:siobounds}, \ref{prop:solidestimates} and \ref{prop:eqaccross0},
that $\Gamma_{(t_0,x_0)}(t,x)$ is a fundamental solution with the stated bounds, using the correspondence between $\divv A^*\nabla$ and $\pd_t+ D\tB$ from Section~\ref{sec:funccalc}.

By Definition~\ref{def:fundsolpm}, we have for all $v_0\in\tE^+_A L_2$ and almost all $(t,x)\in \R^{1+n}_+$ the identity
$$
  \int_{\R^n} \big(\nabla_{A^*} \Gamma^{i}_{(t,x)}(0,y), N v_0(y) \big) dy= (e^{-t BD}v_0)^i_\no(x),
$$
where $\Gamma^{i}= (\Gamma^{ji})_j$.
Now let $v_0= \tE^+_A h=\tE^+_A [h, 0]^t$ for some scalar $h\in L_2(\R^{n};\C^m)$, or equivalently normal vector field $[h, 0]^t\in L_2(\R^{n};\C^{m(1+n)})$.
We then obtain
$$
  \int_{\R^n} \big(\nabla_{A^*} \Gamma^{i}_{(t,x)}(0,y), N \tE^+_A h(y) \big) dy= (e^{-t BD}\tE^+_A h)^i_\no(x). 
$$
Using the duality from Proposition~\ref{prop:dualityandintertw}, the left hand side is
\begin{multline*}
\int_{\R^n} \big(E_{A^*}^-\nabla_{A^*} \Gamma^i_{(t,x)}(0,y), N  h(y) \big) dy=
\int_{\R^n} \big(\nabla_{A^*} \Gamma^i_{(t,x)}(0,y), N h(y) \big) dy \\
= -\int_{\R^n} \big(\pd_{\nu_{A^*}} \Gamma^i_{(t,x)}(0,y), h(y) \big) dy,
\end{multline*}
whereas the right hand side is $( b_t(BD) h)^i_\perp (x)$.
This proves the theorem.
\end{proof}

\section{Fundamental solution for $t$-dependent coefficients}   \label{sec:tdep}

In this section, we prove Theorem~\ref{thm:2} and show some further estimates of the constructed fundamental solutions.
We assume throughout this section that $n\ge 2$ and $m\ge 1$, and that $A_0$ and $A$ are as in the hypothesis of Theorem~\ref{thm:2},
where we choose $\epsilon>0$ small enough so that the De Giorgi--Nash local H\"older estimates \eqref{eq:dGN}, or equivalently \eqref{eq:Morrest}, hold for $A$- and for $A^*$-solutions,
and that $A$ is accretive.
Note that in this section we allow $A_0$ and $A$ to depend on all $n$ variables.
As in the proof of Theorem~\ref{thm:1}, we write $\Gamma^i=(\Gamma^{ji})_j$ and suppress the index $j$, and sometimes also $i$.

\begin{proof}[Proof of Theorem~\ref{thm:2}]
(i)
Define, in $\R^{1+n}$, $t$-independent coefficients
$$
  \tilde A(t,x)[f_\no, f_\ta]^t:= [f_\no, A(x)f_\ta],
$$
so that $\tilde A(t,x)= \tilde A(x)$.
Our aim is to construct a fundamental solution for $A$ on $\R^n$ from the already constructed fundamental solution for $\tilde A$ on $\R^{1+n}$, by integrating away the auxiliary variable $t$.
We assume that $\epsilon>0$ is small enough so that $\tilde A$ is accretive and that De Giorgi--Nash local H\"older estimates \eqref{eq:dGN}, or equivalently \eqref{eq:Morrest}, hold for $\tilde A$- and for $\tilde A^*$-solutions.

In particular this means that the hypothesis of Theorem~\ref{thm:1} is satisfied for $\tilde A^*$,
giving a fundamental solution $\widetilde\Gamma_{(0,x_0)}\in W_{1,\loc}^1(\R^n;\C^m)$ with pole at $(0,x_0)$ to $\divv \tilde A \nabla$ in $\R^{1+n}$ with estimates
$$
  \int_{|x-x_0|>R}  |\nabla \widetilde\Gamma_{(0,x_0)}(t,x)|^2 dx\lesssim (R+|t|)^{-n},
$$
for all $R>0$, $t\in\R$ and $x_0\in\R^n$.

(ii)
Assume first that $n\ge 3$.
Define 
$$
  g_t^i(x):=\nabla_{t,x} \widetilde \Gamma^i_{(0,x_0)}(t,x),
$$
so that
$
 \int_{-\infty}^\infty \|g_t\|_{L_2(R<|x|<2R)} dt\lesssim \int_0^\infty \frac {dt}{(R+t)^{n/2}}\lesssim R^{1-n/2}.
$
Thus 
$$
  g^i(x):= \int_{-\infty}^\infty (g^i_t(x))_\ta dt
$$
converges in $L_2(R<|x|<2R)$, and we have $\|g\|_{L_2(R<|x|<2R)}\lesssim R^{1-n/2}$ so that 
$g\in L_1^\loc(\R^n; \C^{mn})$. 
It suffices to show that $(\divv A g^i)^j=\begin{cases} \delta_{x_0}, & j=i, \\ 0, & j\ne i, \end{cases}$ and $\curl g^i=0$ in $\R^n$-distributional sense.
The latter is clear from the definition of $g$. To prove the former, let $\phi\in C_0^\infty(\R^n;\C^m)$.
Let $\eta\in C_0^\infty(\R)$ be such that $\eta=1$ for $|t|<1$ and $\eta=0$ for $|t|>2$, and let $\eta_T(t):= \eta(t/T)$.
Consider the integral
$$
  I_T:=\iint_{\R^{1+n}}( g^i_t(x), \tilde A^*(x)\nabla_{t,x}(\phi(x)\eta_T(t)) )dtdx= -\phi^i(x_0).
$$
Then 
\begin{multline*}
  I_T= \iint ( (g^i_t(x))_\no, \phi(x) ) \pd_t\eta_T(t) dtdx -
  \iint( (g^i_t(x))_\ta, A^*(x)\nabla \phi(x)) (1-\eta_T(t)) dtdx\\+ \int( g^i, A^*\nabla\phi)dx
  =: II_T- III_T+ \int( g^i, A^*\nabla\phi)dx.
\end{multline*}
The estimates $\|g_t\|_2\lesssim t^{-n/2}$ proves that 
$II_T\to 0$ and $III_T\to 0$ as $T\to\infty$, so that
$\int( g^i, A^*\nabla\phi)dx= -\phi^i(x_0)$.

(iii)
Now let $n=2$.
We claim that in this case
$$
  \sup_{R>0}\int_{|t|>R} \|g_t\|_{L_2(|x|<R)} dt<\infty.
$$
From this claim, it will follow that 
$$
  \|g\|_{L_2(R<|x|<2R)}\lesssim \int_0^{2R} \frac {dt}{R+t}+ \int_{2R}^\infty \|g_t\|_{L_2(|x|<2R)} dt\lesssim 1
$$
and $II_T\to 0$ and $III_T\to 0$ as $T\to\infty$ as in (ii).
To prove the claim, we apply the estimate \eqref{eq:Morrest} to the solution 
$\divv \tilde A \nabla \widetilde \Gamma^i_{(0,x_0)}=0$ in $\sett{(t,x)}{\max(|x|, |t-T|)<|T|/2}\supset \sett{(t,x)}{\max(|x|, |t-T|)<R}$ for $|T|>2R$.
We obtain
\begin{multline*}
  \int_{T-R}^{T+R} \|g_t\|_{L_2(|x|<R)} dt\lesssim \sqrt R \| g \|_{L_2(|x|, |t-T|<R)}
  \lesssim \sqrt R (R/T)^{1/2+\mu} \| g \|_{L_2(|x|, |t-T|<|T|/2)}\\
  \lesssim
  \sqrt R (R/|T|)^{1/2+\mu} \left(\int_{|t-T|<|T|/2} t^{-2} dt \right)^{1/2}\lesssim (R/|T|)^{1+\mu}.
\end{multline*}
From this it follows that 
$$
  \int_{|t|>R} \|g_t\|_{L_2(|x|<R)} dt\lesssim \sum_{k=1}^\infty (R/kR)^{1+\mu}\lesssim 1.
$$
This completes the proof of the theorem.
\end{proof}

\begin{prop}
Under the hypothesis of Theorem~\ref{thm:2}, and suitable choices of integration constants, the following holds.

\begin{itemize}
\item[{\rm (i)}] The gradient of the fundamental solution to $\divv A\nabla$ has estimates
$$
  \int_{B(z,r)} |\nabla \Gamma_x(y)|^2 dy\lesssim r^{n-2+2\mu} |z-x|^{4-2n-2\mu}
$$
for $2r<|z-x|$ and some $\mu>0$.
\item[{\rm (ii)}] If $n\ge 3$, then the fundamental solution to $\divv A\nabla$ has point wise estimates
$$
  |\Gamma_x(y)|\lesssim |y-x|^{2-n},\qquad y\ne x,
$$
and H\"older estimates
$$
  |\Gamma_x(y')-\Gamma_x(y)|\lesssim \left(\tfrac{|y'-y|}{|y-x|}\right)^\alpha |y-x|^{2-n},\qquad |y'-y|<|y-x|/2.
$$
\item
[{\rm (iii)}] If $n=2$, then the fundamental solution to $\divv A\nabla$ has point wise estimates
$$
  |\Gamma_x(y)|\lesssim 1+\big|\ln|y-x|\,\big|,\qquad y\ne x,
$$
and H\"older estimates
$$
  |\Gamma_x(y')-\Gamma_x(y)|\lesssim \left(\tfrac{|y'-y|}{|y-x|}\right)^\alpha (1+\big|\ln|y-x|\,\big|),\qquad |y'-y|<|y-x|/2.
$$
\end{itemize}
\end{prop}

\begin{proof}
(i) 
This follows by from \eqref{eq:keyestimateofgradfund} and \eqref{eq:Morrest}.

(ii)
For $R>0$, consider the mean values
$$
  A_R:= \frac 1{(2^n-1)\sigma_{n-1} R^n}\int_{R<|y-x|<2R} \Gamma_x(y) dy.
$$
We obtain from Poincar\'e's inequality, with means over the inner/outer halves of the annuli, and \eqref{eq:keyestimateofgradfund} that
\begin{multline*}
  |A_{2R}-A_R|\lesssim  \frac 1{R^n}\left| \int_{R<|y-x|<4R}\big( \Gamma_x(y) -A_R\big) dy\right| + 
   \frac 1{R^n}\left| \int_{R<|y-x|<4R}\big(\Gamma_x(y)-A_{2R}\big)dy \right| \\
   \le \frac 1{R^n}\int_{R<|y-x|<4R}  \big( |\Gamma_x(y) -A_R  | + |\Gamma_x(y) -A_{2R} |\big) dy\\
 \lesssim \left(\frac 1{R^n}\int_{R<|y-x|<4R}  \big( |\Gamma_x(y) -A_R  |^2 + |\Gamma_x(y) -A_{2R} |^2\big) dy\right)^{1/2}\\
 \lesssim \left(R^{2-n}\int_{R<|y-x|<4R}  |\nabla\Gamma_x(y)|^2 dy\right)^{1/2}\lesssim R^{2-n}.
\end{multline*}
If $n\ge 3$, we obtain the estimate
\begin{equation}   \label{eq:dyadicmeanest}
  |A_{2^j}-A_{2^k}|\lesssim (2^k)^{2-n},
\end{equation}
for all $j,k\in \Z$ with $j>k$. In particular $\lim_{j\to \infty}A_{2^j}$ exists. Choosing the constant of integration, we assume that this limit is zero. This gives
$$
  |A_{R}|\lesssim R^{2-n},\qquad \text {for all }R>0,
$$
and again by Poincar\'e's inequality and \eqref{eq:keyestimateofgradfund} that
\begin{multline*}
  \left( \frac 1{R^n}\int_{R<|y-x|<2R} |\Gamma_x(y)|^2 dy\right)^{1/2}\lesssim |A_R|+  \left(\frac 1{R^n}\int_{R<|y-x|<2R} |\Gamma_x(y)-A_R|^2 dy\right)^{1/2} \\
   \lesssim |A_R|+  \left(R^{2-n}\int_{R<|y-x|<2R} |\nabla\Gamma_x(y)|^2 dy\right)^{1/2}\lesssim R^{2-n}.
\end{multline*}
Using the Moser local boundedness estimate \eqref{eq:Mos} and the De Giorgi--Nash local H\"older estimate \eqref{eq:dGN}, this proves the estimates (ii).

(iii)
If $n=1$, the equation~\eqref{eq:dyadicmeanest} becomes
$$
  |A_{2^j}-A_{2^k}|\lesssim j-k.
$$
Choosing the constant of integration so that $|A_1|\lesssim 1$, this gives
$$
  |A_R|\lesssim 1+\big|\ln R\,\big|.
$$
The point wise estimates (iii) then follows as in (ii).
\end{proof}

\section{The gradient of the single layer potential operator}   \label{sec:singlelayer}

We end this paper by deriving results for the single layer potential operator
$$
   \slp_t h^i(x) = \int_{\R^n} \Gamma^{ij}_{(0,y)}(t,x) h^j(y) dy,
$$
where $\Gamma$ here denotes the fundamental solution for $\divv A\nabla$.
Recall that the Neumann problem, with boundary datum $\varphi$, is solved through the ansatz $u(t,x):= \slp_t h(x)$, where the auxiliary boundary function $h$ solves the equation
$$
   \lim_{t\to 0^+} \pd_{\nu_A} \slp_t h = \varphi.
$$
We prove the following result for the single layer potential operator, analogous to Theorem~\ref{thm:1}
for the double layer potential operator.

\begin{thm}  \label{thm:slp}
Assume the hypothesis of Theorem~\ref{thm:1}, with $A$ replaced by $A^*$, so that $\Gamma$ now denotes the fundamental solution for $\divv A\nabla$.
Then 
\begin{equation}   \label{eq:singlelayrepr}
   \nabla_A \int_{\R^n} \Gamma^{ij}_{(0,y)}(t,x) h^j(y) dy= (e^{-tDB} E_A^+ h)^i(x)
\end{equation}
holds for almost all $(t,x)\in \R^{1+n}_+$ and all scalar functions $h\in L_2(\R^n;\C^m)$.
We here identify $h$ with a normal vector field $h\in L_2(\R^n;\C^{m(1+n)})$ on the right hand side.
\end{thm} 

This theorem allows us to transfer known results for the conormal gradient of the single layer potential operator 
$$
  \nabla_A\tslp_t^A h^i= \nabla_A\tslp_t h^i :=(e^{-tDB} E_A^+ h)^i,\qquad t>0,
$$
defined through functional calculus, to the conormal gradient of the single layer potential operator
$$
 \nabla_A\slp_t^A h^i= \nabla_A\slp_t h^i:=\nabla_A \int_{\R^n} \Gamma^{ij}_{(0,y)}(t,\cdot) h^j(y) dy, \qquad t>0,
$$
defined classically as an integral operator.
The following is a list of such known results for $\nabla_A\tslp^A_t$ which extends the results for $\nabla_A\slp^A_t$ from \cite{AAAHK, HKMP}.

\begin{itemize}
\item We have estimates
$$
\sup_{t>0} \|\nabla_A\tslp_t h\|_2^2+  \int_0^\infty \|\pd_t \nabla_A\tslp_t h \|_2^2\, tdt + \| \tN (\nabla_A\tslp_t h)\|_2^2\lesssim \|h\|^2,
$$
for any system with bounded and accretive coefficients $A$.
In particular, the implicit constant in this estimate depends only on $\|A\|_\infty$, $\kappa_A$, $n, m$, but not on the De Giorgi--Nash--Moser constants.

\item
For any system with bounded and accretive coefficients $A$, the operators $\nabla_A\tslp_t$ converge strongly in $L_2$ and there exists an $L_2(\R^n;\C^m)$ bounded operator $\nabla_A\tslp$ such that
$$
  \lim_{t\to 0^+}\|\nabla_A\tslp_t h- \nabla_A\tslp h\|_2=0,\qquad\text{for all } h\in L_2(\R^n;\C^m).
$$

\item
The map
$$
  \{ \text{accretive } A\in L_\infty(\R^n;\mL(\C^{m(1+n)}))\} \ni A\mapsto \nabla_A\tslp_t^A\in \mL(L_2)
$$
is a holomorphic map between Banach spaces. 
In particular, $\nabla_A\tslp_t^A\in \mL(L_2)$ depends locally Lipschitz continuously on $A\in L_\infty(\R^n;\mL(\C^{m(1+n)}))$, and therefore invertibility of $\lim_{t\to 0^+}\pd_{\nu_A} \tslp_t^A$ is stable under small $L_\infty$ perturbations of $A$.

\item
The operator $\pd_{\nu_A} \tslp_t^A\in \mL(L_2(\R^n;\C^{m}))$ is invertible when $A$ is Hermitian, $(A^{ij})^*= A^{ji}$, when 
$A$ is constant, $A(x)= A$, and when $A$ is of block form, $A^{ij}_{\no\ta}=0= A^{ij}_{\ta\no}$.
The counter example to invertibility of $\tdlp$ mentioned in the introduction applies also to $\lim_{t\to 0^+}\pd_{\nu_A} \tslp_t$.
\end{itemize}

\begin{proof}
In the classical case of integral operators, the conormal derivative of the single layer potential is dual to the double layer potential operator.
Similarly, the proof of Theorem~\ref{thm:slp} is by duality.
  We note from Definition~\ref{def:fundsolpm}  that 
$$
   \int_{\R^n} \Big( \nabla_A \Gamma^{ij}_{(-t, y)}(0,x), Nv^i_0(x) \Big) dx=- ( e^{t\widehat{A^*}D}\tE_{A^*}^-v_0)^j_\perp(y)
$$
for all $t>0$ and $v_0\in L_2$.
Integrate this equation against a scalar/normal vector field $h\in L_2(\R^n;\C^{m})$ to obtain
$$
 \int_{\R^n}\left(  \int_{\R^n} \nabla_A \Gamma^{ij}_{(-t, y)}(0,x) h^j(y) dy, Nv^i_0(x) \right) dx=- ( h,e^{t\widehat{A^*}D}\tE_{A^*}^-v_0).
$$
Since $\widehat{A^*}D= NB^*ND= N(-B^*D)N$, we obtain
$$
 \int_{\R^n}\left( \nabla_A \int_{\R^n}  \Gamma^{ij}_{(0, y)}(t,x) h^j(y) dy, Nv^i_0(x) \right) dx=(e^{-tDB}E_{A}^+ h, Nv_0).
$$
Since $v_0\in L_2$ is arbitrary, this proves \eqref{eq:singlelayrepr}.
\end{proof}

\bibliographystyle{acm}

\end{document}